\documentclass[a4paper]{amsart}
\usepackage{amssymb,amscd,verbatim,latexsym}
\usepackage{enumerate}
\usepackage{tikz-cd}
\usepackage{tikz}
\usepackage[all,pdf]{xy}

\usepackage[T1]{fontenc}
\usepackage{ae} %
\usepackage{aecompl}%
\usepackage[american]{babel}
\usepackage[utf8]{inputenc}
\usepackage{csquotes}
\usepackage{stmaryrd}

\usepackage{lmodern}
\usepackage{microtype}

\usepackage[ocgcolorlinks]{hyperref}
\hypersetup{
  colorlinks = true,
  allcolors = [rgb]{0.15, 0.37, 0.41},
  linktocpage = true}

\usepackage[mathscr]{eucal}

\usepackage{color}

\newtheorem{theorem}{Theorem}[section]

\newtheorem{proposition}[theorem]{Proposition}
\newtheorem{corollary}[theorem]{Corollary}

\newtheorem{lemma}[theorem]{Lemma}

\theoremstyle{definition}

\theoremstyle{remark}

\author[C. Dell'Aiera]{Cl\'ement Dell'Aiera}
\address{Department of Mathematics, UMPA, ENS Lyon
46 allée d’Italie
69342 Lyon Cedex 07
FRANCE}
\email{clement.dellaiera@ens-lyon.fr}

\thanks{}

\date{\today}
\title[Coarse geometry of Hecke pairs and the Baum-Connes conjecture]{Coarse geometry of Hecke pairs and the Baum-Connes conjecture} 

\begin{document}

\maketitle

\begin{abstract} We study Hecke pairs using the coarse geometry of their coset space and their Schlichting completion. We prove new stability results for the Baum-Connes and the Novikov conjectures in the case where the pair is co-Haagerup. This allows to generalize previous results, while providing new examples of groups satisfying the Baum-Connes conjecture with coefficients. For instance, we show that for some $S$-arithmetic subgroups of $Sp(5,1)$ and $Sp(3,1)$ the conjecture with coefficients holds.
\end{abstract}

\tableofcontents
\section{Overview and statement of the results}

The Baum-Connes conjecture predicts that the $K$-theory groups of the reduced $C^*$-algebra of a locally compact group, which is the norm closure of the complex algebra generated by the left regular representation, is isomorphic to the equivariant $K$-homology of the group's classifying space for proper actions.
One of its most spectacular applications is the descent principle, that allows to derive the Novikov conjecture from a certain form of injectivity of the Baum-Connes assembly map. See section \ref{BC} for a reminder with references for both statements.

The conjectures are known to hold in many cases, and the Baum-Connes conjecture has various stability properties. For instance, groups with the Haagerup property satisfy the Baum-Connes conjecture with coefficients, and the conjecture is stable by extensions. Moreover if a group acts by isometries on a tree with stabilizers that satisfy the Baum-Connes conjecture, then so does the group. Recall that the Haagerup property can be defined as the existence of a metrically proper action on a real affine Hilbert space by isometries. 

This leads to the following question: if a group acts on a real affine Hilbert space by isometries, suppose that one orbit is a proper subspace, but possibly with infinite isotropy subgroups. Can we deduce the Baum-Connes conjecture for the group if the stabilizer satisfies it?

In this setting, the typical stabiblizer of the proper orbit is co-Haagerup in the ambient group, and this forces the subgroup to be almost normal, in the sense that it is commensurable to any of its conjugate. We answer by the affirmative to the question with the following.

\begin{theorem}
Let $\Lambda < \Gamma$ be a co-Haagerup subgroup of a discrete countable group. Then, if all subgroups of $\Gamma$ containing $\Lambda$ as a subgroup of finite index satisfy the Baum-Connes conjecture with coefficients, so does $\Gamma$.
\end{theorem}

Being almost normal is weaker than co-Haagerup. It is actually equivalent to $\Gamma / \Lambda$ being of bounded geometry, if we equip $\Gamma / \Lambda$ with the metric induced from a left proper metric on $\Gamma$. With this in mind, we deduce the following from the theorem.

\begin{corollary}
Let $\Lambda < \Gamma$ be a Hecke pair. If $\Lambda$ and $\Gamma / \Lambda$ admit a coarse embedding into a Hilbert space, then $\Gamma$ satisfies the Novikov conjecture.
\end{corollary}

The paper is organized as follows. The second section gives a geometric characterization of Hecke pairs: a subgroup is almost normal iff the coset space with the quotient metric is of bounded geometry. In the third section, we review the construction of the Schlicting completion of a Hecke pair, a totally disconnected locally compact group that acts as a replacement of the quotient group when the subgroup is only almost normal, and prove that a subgroup is co-Haagerup iff the corresponding Schlichting completion has Haagerup's property. Here, we use implicitely that co-Haagerup subgroups are almost normal. The fourth section is devoted to the proof of the main theorem, and the fifth section to the proof of the corollary. In the last section, we apply these results to establish that the Baum-Connes conjecture with coefficients holds for some countable discrete groups. The first examples recover previous known results with a different proof, the second examples are, to the author's knowledge, new. For instance, we have the following.

\begin{corollary}
Let $G$ be an absolutely simple algebraic group over $\mathbb Q$ such that groups containing $G(\mathbb Z)$ as a subgroup of finite index satisfy the Baum-Connes conjecture with coefficients. Let $p$ be a prime number, and supppose that the $\mathbb Q_p$-rank of $G$ is $1$, then $G\left(\mathbb Z[\frac{1}{p}]\right)$ satisfies the Baum-Connes conjecture with coefficients.  
\end{corollary}   

This can be applied when $G(\mathbb Z)$ is a uniform lattice in $Sp(3,1)$ and $Sp(5,1)$ (or $SO(n,1)$ for $n=5,7,9$) since these are Gromov hyperbolic groups, and thus satisfies the Baum-Connes conjecture with coefficients.\\

\textbf{Acknowledgments.} The author would like to thank Erik Guentner for helpful discussions on this topic, and for suggesting the question in the first place. He is also indebted to Hervé Oyono-Oyono for comments on the first draft. \\


\section{Coarse geometry and Hecke pairs}

Let $\Gamma$ be a discrete group. A subgroup $\Lambda < \Gamma$ is \textit{almost normal} if one of the following equivalent conditions is satisfied: 
\begin{itemize}
\item for every $\gamma\in\Gamma$, $\Lambda$ and $\Lambda^\gamma = \gamma \Lambda \gamma^{-1}$ are commensurable (i.e. they contain a common subgroup of finite index),
\item the index $[\Lambda :\Lambda \cap \Lambda^\gamma]$ is finite, for every $\gamma\in \Gamma$,  
\item the left action of $\Lambda$ on $\Gamma / \Lambda$ has finite orbits,
\item every double coset $\Lambda s \Lambda$ is a finite union of cosets $\gamma\Lambda$.  
\end{itemize}

In this case, we call $(\Gamma,\Lambda)$ a Hecke pair. The equivalence is easily seen since the cardinal of the orbit $\Lambda g \Lambda $ is the index $[\Lambda : \Lambda\cap g\Lambda g^{-1}]$. Let us fix a left $\Gamma$-invariant metric on $\Gamma$, given by a proper length $|\cdot |$. We endow $X=\Gamma/ \Lambda$ with the left $\Gamma$-invariant metric 
$$d(s\Lambda , t\Lambda) = \inf_{\lambda, \lambda'\in \Lambda} |\lambda s^{-1}t\lambda' |.$$ 

Recall that a metric space $(X,d)$ is of \textit{bounded geometry} if for every $r>0$, $\sup_{x\in X} |B(x,r)|$ is finite. 

\begin{proposition}
The coset space $X=\Gamma /\Lambda$ is of bounded geometry iff $(\Gamma,\Lambda)$ is a Hecke pair.  
\end{proposition} 

\begin{proof}
The metric being left invariant and the action transitive, it is enough to show that any ball of finite radius is finite. But $d(g\Lambda, \Lambda) \leq r$ iff 
$$g\in \cup_{|\gamma|\leq r} \Lambda \gamma \Lambda .$$
$\Gamma$ is of bounded geometry, so that the latter is a finite union of double cosets $\Lambda \gamma \Lambda$, themselves being a finite union of left $\Lambda$-cosets by almost normality. 

Now, if $X$ is of bounded geometry, $\Gamma$ acts by isometries, by left invariance of the metric. As $\Lambda$ stabilizes the base point, it stabilizes all spheres, and thus its orbits are contained in those, which are finite.  
\end{proof}

This gives a large class of examples of Hecke pairs. Let $\Gamma$ be a discrete group acting by isometries on a locally finite metric space, then any stabilizer is almost normal. For instances, groups acting by isometries on locally finite trees, such as HNN extensions and amalgated free products, have almost normal subgroups.
\begin{itemize}
\item If $BS(m,n) = \langle a,b | a^{-1}b^m a = b^n\rangle$ is the Baumslag-Solitar group, then $\mathbb Z \cong \langle b\rangle $ is an almost normal subgroup.     
\item $SL(2,\mathbb Z)$ is almost normal in $SL(2,\mathbb Z[\frac{1}{p}])$, by considering its restricted action on the Bass-Serre tree of $SL(2,\mathbb Q_p)$.
\end{itemize}

Other examples do not readily come from isometric actions. The previous proposition gives a geometric interpretation to these pairs.     
\begin{enumerate}
\item If $\Gamma$ is a discrete group acting on a set $X$, and $Y\subset X$ a commensurated subset, i.e. the symmetric difference $|Y\Delta \gamma Y|<\infty$ for every $\gamma\in \Gamma $, and $F$ a finite group, then $\bigoplus_Y F$ is almost normal in the (generalized) wreath product $F\wr_X \Gamma = (\bigoplus_X F)\rtimes \Gamma$. If one specifies $\Gamma =\mathbb Z$ and $Y=\mathbb N\subset X=\mathbb Z$, we get an almost normal subgroup of the Lamplighter group.\label{wrExpl}
\item $SL(n,\mathbb Z)$ is almost normal in $SL(n,\mathbb Q)$. 
More generally, arithmetic lattices in global fields have commensurated subgroups: if $F$ is a global field, and $\mathcal O$ its ring of integers, let $G$ be an absolutely simple, simply connected algebraic group over $F$. Let $S$ and $S'$ be sets of inequivalent valuations on $F$, containing all archimedean ones, and such that $S'\subset S$. We denote by $\mathcal O_S$ the ring of $S$-integers in $F$. A $S$-arithmetic group is a subgroup commensurable with $G(\mathcal O_S)$. Then if $\Gamma$ is a $S$-arithmetic group, any $S'$-arithmetic group $\Lambda$ is almost normal in $\Gamma$.  
\end{enumerate}

\section{The Schlichting completion and coarse embeddings} %

Let $(\Gamma ,\Lambda)$ be a Hecke pair and $X=\Gamma / \Lambda$. There exists a locally compact totally discontinuous Hecke pair $(G,K)$ where $K$ is a compact open subgroup of $G$, and a homomorphism $\sigma : \Gamma \rightarrow G$ with dense image satisfying $\sigma^{-1}(K) = \Lambda$, hence inducing isomorphisms $\Gamma/\Lambda\cong G/K$ and $\Lambda \backslash\Gamma/\Lambda\cong K \backslash G/K$. This construction was introduced by Schlichting in \cite{schlichting}, and used extensively by Tzanev  in \cite{tzanev2003hecke}.   

Let us recall the construction: we endow the group of permutations $\mathfrak S(X)$ with the topology induced from pointwise convergence in the space of maps from $X$ to $X$. It is a standard fact that this makes $\mathfrak S(X)$ a Polish group. We denote by $\sigma: \Gamma \rightarrow \mathfrak S(X) $ the representation by permutation, and by $G$ (respectively $K$) the closure of the image of $\Gamma$ (respectively $\Lambda$) by $\sigma$. These are totally discontinuous groups. 

From this follows that $K$ is compact open if $(\Gamma,\Lambda)$ is a Hecke pair, thus $G$ is locally compact. Indeed, $K$ is a closed subgroup of the group $$\prod_{[g]\in \Lambda \backslash X} \mathfrak S(\Lambda g \Lambda/\Lambda),$$ which is compact as a  product of finite groups (the topology of pointwise convergence coincides with the product topology). It is also the stabilizer of a point, $K=Stab_G(\Lambda)$, hence it is open since the finite intersections of stabilizers form a basis for the topology of pointwise convergence. The group $G$ thus has a compact open neighborhood of the identity.

The following points are important: 
\begin{itemize}
\item If $\Lambda$ is normal, the pair $(G,K)$ is $(\Gamma / \Lambda , 1)$,
\item If $\Lambda$ is finite, then $N=\cap_\gamma \Lambda^\gamma$ is a finite normal subgroup of $\Gamma$ contained in $\Lambda$, and $(G,K) \cong (\Gamma / N, \Lambda /N)$.
\item The definition of a Hecke pairs makes sense if $\Gamma$ is locally compact and $\Lambda$ is open (and closed) in $\Gamma$. Then the previous remarks remain true, if finite is replaced by compact. In general, Hecke pairs in totally disconnected locally compact groups are useful, with almost normal subgroups given by compact open subgroups.
\end{itemize}

We see that the biggest normal subgroup contained in $\Lambda$ is $N=\cap_\gamma \Lambda^\gamma$. We will call $N$ the core of $\Lambda$. The Hecke pair is a substitute for the quotient group in the absence of normality. It is thus natural to focus on \textit{reduced Hecke pairs}, i.e. such that $N$ is trivial. If a pair $(\Gamma, \Lambda)$ is not reduced, its reduced pair will be $(\Gamma/N, \Lambda / N)$. A useful result to identify the Schlichting completion of a Hecke pair is the following.

\begin{lemma}[lemma 3.5 in \cite{shalom2013commensurated}]\label{LmWillis}
Let $(\Gamma,\Lambda)$ a Hecke pair. Suppose there exist a locally compact group $G$, a compact open subgroup $K<G$, and a homomorphism $\psi : \Gamma \rightarrow G$ such that $\psi(\Gamma) $ is dense in $G$ and $\psi^{-1}(K)=\Lambda$. Then the Schlichting completion of $(\Gamma,\Lambda)$ and $(G,K)$ coincide. In particular, that of $\Gamma$ is isomorphic to $G/N$, where $N$ is the largest normal subgroup contained in $K$.  
\end{lemma}

Here are some examples of computation of Schlichting completions. 
\begin{itemize}
\item If $\Lambda = \bigoplus_Y F$ in $\Gamma= F\wr_X G$, as in \ref{wrExpl} with $F$ finite and $Y$ commensurated in $X$, let us define $G = P\rtimes \Gamma$ where 
$$P = (\prod_X F )\oplus ( \bigoplus_{\Gamma\backslash X} F )\subset \prod_\Gamma F.$$
Then $\Gamma \hookrightarrow G $ satisfies the hypothesis of the lemma, with $K= \prod_X F$. The core of $K$ is easily seen to be $N=\prod_{\cap_\gamma \gamma\cdot X} F$, and $G/N$ is the Schlichting completion in that case. Notice that in the case where the intersection of all translates of $X$ is trivial, $N$ also is, so that $G$ is the Schlichting completion of $\Gamma$.  

\item For $p$ prime, the Schlichting completion of $(SL(n,\mathbb Z[\frac{1}{p}]), SL(n,\mathbb Z))$ is $PSL(n,\mathbb Q_p)$, by using $SL(n,\mathbb Z[\frac{1}{p}]) \hookrightarrow SL(n,\mathbb Q_p)$. 
\end{itemize}

This last example is a particular case of a general statement. Let $\mathbb F$ a number field, and $\mathcal O$ its ring of integers. Let $\overline{\mathbb F}$ and $\overline O$ their completions with respect to all non-Archimedian valuations of $\mathbb F$. Then, if $G$ is an algebraic group over $\mathbb F$, $(G(\mathbb F), G(\mathcal O))$ is a Hecke pair, and its Schlichting completion is the reduced pair of their respective completion in $G(\overline{ \mathbb F})$ (see \cite{tzanev2003hecke}, after proposition 4.1).	

Recall that a group is a-T-menable, also called Haagerup's property, if there exists a real valued continuous function on $G$ that is proper and conditionnaly of negative type (see \cite{cherix2001groups}, chapter 1). We also recall that a metric space with bounded geometry
\begin{itemize} 
\item admits a coarse embedding into Hilbert space if there exists a symmetric normalised kernel on $X$ that is conditionnaly of negative type and effectively proper (see \cite{cornulier2012proper}, definition 5.6), 
\item has Yu's property (A) if for every positive numbers $\varepsilon$ and $r$, there exists a symmetric normalised kernel on $X$ of positive type with finite propagation and $(r,\varepsilon)$-propagation (see \cite{willett2006some}, theorem 1.2.4). 
\end{itemize}
Furthermore, a subgroup $\Lambda <\Gamma$ is c\o -Folner iff $\Gamma / \Lambda$ carries a $\Gamma$-invariant mean. Exactness of a locally compact group is defined as exactness of the reduced crossed-product. It is known to be equivalent to \textit{amenability at infinity}, that is that $G$ admits an amenable action on some compact Hausdorff space (see \cite{brodzki2017exactness}).  

From these definitions (that are actually theorems), we see that a discrete group is a-T-menable iff it admits a $\Gamma$-equivariant coarse embedding into Hilbert space, and that it is amenable iff it satisfies property (A)'s condition with the kernel being $\Gamma$-equivariant.
\begin{proposition}\label{CEHATM}
With the notation above:
\begin{itemize}
\item $X$ admits a $\Gamma$-equivariant coarse embedding into a $\Gamma$-Hilbert space iff $G$ has Haagerup's property,
\item $X$ admits a coarse embedding into a Hilbert space iff the action of $G$ on $\beta X$ is a-T-menable. 
\item $\Lambda $ is co-Fl\o ner in $\Gamma$ iff $G$ is amenable,
\item $X$ has Yu's property (A) iff $G$ is exact.  
\end{itemize}
\end{proposition}

\begin{proof}
The key fact is the correspondance between kernels on $X$ and on $G$. Indeed, the map quotient map $G\rightarrow X$ induces a map that takes kernels on $X$ to kernels on $G$, respects properness and, if the original kernel is $\Gamma$-invariant, its image will be $G$-invariant. Thus, if we have a conditionally negative type $\Gamma$-equivariant metrically proper kernel on $X$, we get a continuous conditionally negative type proper function on $G$. 

For the converse, if we have a continuous conditionally negative type proper function $\phi : G\rightarrow \mathbb R$, then 
$$\varphi(sK,tK) = \int_{K}\int_{K}\phi(k_1s^{-1}t k_2)dk_1dk_2$$
defines a conditionally negative type $\Gamma$-equivariant metrically proper kernel on $X$.  

Remark that these two correspondances respects the support in the following sense: $supp \varphi \subset \{ (x,y)\in X\times X : d(x,y)\leq r\}$ iff $supp \phi \subset \cup_{|s|\leq r} K\sigma(s)K$. Thus kernels supported in an entourage of $X$ correspond to compactly suppported kernels on $G$. This gives the two last points.


\end{proof}

The action of $\Gamma$ on $\beta X$ extends to a continuous action of $G$, and the coarse groupoid of $X$ (see \cite{skandalis2002coarse}) is naturally isomorphic to  a quotient of $$\beta X \rtimes G$$ with compact kernel. In the case of a normal subgroup, we recover the fact that $\mathcal G(X) \cong \beta Q\rtimes Q$ ($Q$ being the quotient group). This implies that $X$ has geometric property (T) (see \cite{willett2014geometric}) iff the action of $G$ on $\beta X$ has dynamical property (T) (see \cite{dell2018topological}), and that the asymptotic dimension of $X$ is the dynamical asymptotic dimension of $G$ acting on $\beta X$ (see \cite{guentner2017dynamic}).

Let $H$ be a subgroup of $\Gamma$. Recall that $H$ is co-Haagerup in $\Gamma$ if there exists a proper $\Gamma$-invariant kernel of conditionally negative type on $\Gamma/H$, and is co-F\o lner if $\Gamma / H$ carries a $\Gamma$-invariant mean. In general, co-F\o lner subgroups are not co-Haagerup (see example 6.1 of \cite{cornulier2012proper}), but in the case of Hecke pairs, it follows from the previous proposition that it this implication holds. Moreover, if $\Lambda < \Gamma$ is co-Haagerup, it is a Hecke pair (see example 6.1 and proposition B.2 of \cite{cornulier2012proper}) and the converse obviously does not hold. We easily see that Hecke pairs which admits a $\Gamma$-equivariant coarse embedding into a Hilbert space are thus exactly the co-Haagerup subgroups.   

\section{Stability of Baum-Connes conjecture for Hecke pairs} %
\label{BC}
The goal of this section is to prove the following.

\begin{theorem}\label{THM1}
Let $(\Gamma , \Lambda)$ be a Hecke pair and $A$ a $\Gamma$-algebra. If every subgroup of $\Gamma$ that is commensurable with $\Lambda$ satisfy the Baum-Connes conjecture with coefficients, and $\Gamma / \Lambda$ admits a $\Gamma$-equivariant coarse embedding into Hilbert space, then $\Gamma$ satisfies the Baum-Connes conjecture with coefficients.
\end{theorem}

This generalizes previous results:
\begin{itemize}
\item If $\Lambda$ is normal, the theorem reduces to a particular case of Oyono-Oyono's stability result of Baum-Connes by extensions (see \cite{oyono2001ext}), namely the case where the quotient is a-T-menable.  
\item If $\Gamma/\Lambda$ embeds into a locally finite tree, the theorem reduces to Oyono-Oyono's stability result of Baum-Connes for groups acting on trees (see \cite{oyono2001baum}).
\end{itemize}

The theorem relies on Higson-Kasparov result that a-T-menable groups satisfies the Baum-Connes conjecture with coefficients (\cite{higson2001theory}). It implies that if a group admits an action by isometries on a real Hilbert space with an orbit that is proper as a metric space, and the commensurated class of the stabilizer satisfies the Baum-Connes conjecture with coefficients, then the group also does.

If $\Lambda$ and $\Gamma$ are discrete groups, let us say that $\Gamma$ is a co-Haagerup extension if $\Lambda$ is isomorphic to an almost normal subgroup of $\Gamma$ such that the resulting quotient equivariantly coarsely embeds into a Hilbert space. We define $\mathcal C$ to be the smallest class of groups containing a-T-menable groups and Gromov hyperbolic groups, that is closed under co-Haagerup extensions. The theorem implies the following.

\begin{corollary}\label{COR1}
All groups of class $\mathcal C$ satisfies the Baum-Connes conjecture with coefficients.
\end{corollary}  

See section \ref{EXPL} for a discussion on the class $\mathcal C$.


\subsection{Preliminaries}  

We first establish general conventions and notations, then give an overview of the proof. 

Let $G$ be a locally compact group, and $A$ a $G$-algebra, by which we mean a $C^*$-algebra endowed with an action $\alpha : G\rightarrow Aut(A)$ of $G$ by $*$-automorphisms. We suppose as usual that $g\mapsto \alpha_g(a) $ is continuous for every $a\in A$. We will often leave $\alpha$ implicit. We will denote the reduced-crossed product by $A\rtimes_r G$. 

We say that $G$ satisfies the Baum-Connes conjecture with coefficients in $A$ if the Baum-Connes assembly map 
$$\mu_{G,A} : K_\bullet^{top}(G,A) \rightarrow K_\bullet (A\rtimes_r G)$$ 
is an isomorphism (see \cite{baum1994classifying} for a definition). For convenience, we will write $BC(G,A)$ for this statement. If the coefficients are not specified, they are meant to be the complex numbers with trivial $G$-action. The conjecture with coeffcients means that $BC(G,A)$ holds for all $G$-algebras $A$.

The Baum-Connes conjecture with coefficients is known to hold for
\begin{itemize}
\item a-T-menable groups (Higson-Kasparov \cite{higson2001theory})
\item Gromov hyperbolic groups (Lafforgue \cite{lafforgue2012conjecture})
\item groups acting on trees with a-T-menable stabilizers (Oyono-Oyono \cite{oyono2001baum}).
\end{itemize} 

Counterexamples with non trivial coefficients are known (see \cite{higson2002counterexamples}). With complex coefficients, the Baum-Connes conjecture is still open, and it also holds for discrete cocompact subgroups of rank one real Lie groups or $SL(3,F)$ for a local field $F$ (Lafforgue \cite{lafforgue2002k}).

In the case of a product group $G=G_1\times G_2$, $A\rtimes_r G_1$ is a $G_2$-algebra, $A\rtimes_r G\cong (A\rtimes_r G_1)\rtimes_r G_2$, and the assembly map can be factored by a partial assembly map. Indeed, let  
$$\mu_{G_1,A}^{(G_2)} : K_\bullet^{top}(G_1\times G_2,A) \rightarrow K_\bullet^{top}(G_2,A\rtimes_r G_1)$$
be the partial assembly map, first defined in \cite{chabert2000baum} (see definition 3.9, or section 2 of \cite{chabert2004going}). Then the following diagram commutes
\[\begin{tikzcd}
K_\bullet^{top}(G,A) \arrow{r}{\mu_{G_1,A}^{(G_2)}}\arrow{rd}{\mu_{G,A}} & K_\bullet^{top}(G_2,A\rtimes_r G_1)\arrow{d}{\mu_{G_2,A\rtimes_r G_1}} \\
                                  & K_\bullet(A\rtimes_r G)
\end{tikzcd}.\]

We will use $BC^{(G_2)}(G_1,A)$ to refer to the statement that $\mu_{G_1,A}^{(G_2)}$ is an isomorphism. 

The second ingredient in the proof is the use of Morita invariance of the Baum-Connes assembly map. In our case, we can restrict to Shapiro's lemma, proved in \cite{chabert2003shapiro}. 

Recall that if $H$ is a closed subgroup of a locally compact group $G$, and $A$ a $H$-algebra with $H$-action $\alpha$, the induced algebra $ind_H^G(A)$ is defined as the sub-$C^*$-algebra of the bounded continuous functions $f:G\rightarrow A$ satisfying $f(gh) = \alpha_h(f(g))$ for every $g\in G, h\in H$, and such that the function $ gH \mapsto \| f(gH) \|$ belongs to $C_0(G/H)$. It is a $G$-$C^*$-algebra with the $G$-action $\alpha_g(f)(s)=f(g^{-1} s)$ for $f\in ind_H^G(A)$ and $g,s\in G$. 

\begin{proposition}[Corollary 0.6 \cite{chabert2003shapiro}] 
Let $H$ be a closed subgroup of a locally compact group $G$, and $A$ a $H$-algebra. Then $BC(G,ind_H^G(A))$ holds iff $BC(H,A)$ does.
\end{proposition}

Our strategy to prove theorem \ref{THM1} is the following.
\begin{enumerate}
\item We realize $\Gamma$ as a closed sugroup of $\Gamma\times G$, where $G$ is the Schlichting completion of the Hecke pair $(\Gamma, \Lambda)$. 
\item We define a transitive continuous action of $\Gamma\times G$ on $G$, with stabilizers isomorphic to $\Gamma$. Shapiro's lemma ensures that $BC(\Gamma,A)$ is equivalent to $BC(\Gamma\times G,C_0(G,A))$.
\item If $\Gamma/\Lambda$ admits a $\Gamma$-equivariant coarse embedding into a Hilbert space, then $G$ is a-T-menable and thus satisfies the Baum-Connes conjecture with coefficients. Factorization by the partial assembly map ensures that it is enough to prove $BC^{(G_2)}(\Gamma, C_0(G,A))$ in order to show $BC(\Gamma\times G,C_0(G,A))$.
\item We show that the Baum-Connes conjecture for all subgroups $L<\Gamma$ containing $\Lambda$ as a subgroup of finite index implies $BC^{(G_2)}(\Gamma, C_0(G,A))$. 
\end{enumerate}

\subsection{Proof}  

Let $A$ be a $\Gamma$-algebra. Define the action of $\Gamma\times G$ on $C_0(G,A)$ by  
$$((\gamma, g)\cdot f)(x) = \gamma\cdot (f(\gamma x g^{-1})).$$

\begin{proposition}
In the above setting, if $\mu_{\Gamma, C_0(G,A)}^{(G)}$ and $\mu_{G,C_0(G,A)\rtimes_r \Gamma }$ are isomorphisms, then $\Gamma$ satisfies the Baum-Connes conjecture with coefficients in $A$.  
\end{proposition}

\begin{proof}
Let $\Gamma\times G$ act on $G$ by $$(\gamma , g)\cdot x = \sigma_\gamma x g^{-1}.$$
The action is transitive, and the stabilizer of $e_G$ is isomorphic to $\Gamma$:
$$Stab(e_G) = \{(\gamma, \sigma_\gamma)\}_{\gamma\in \Gamma}\cong \Gamma.$$
Since $G$ is Hausdorff, the stabilizer is closed, and by corollary 0.6 of \cite{chabert2003shapiro}, for every $\Gamma\times G$ algebra $(A,\alpha)$, 
$$ BC(\Gamma\times G, C_0(G,A)) \Leftrightarrow BC(\tilde \Gamma , A_{|\tilde \Gamma}).$$

We denoted the stabilizer $Stab(e_G)$ by $\tilde \Gamma$ to differentiate it from its isomorphic image $\Gamma$ by the first projection. Here, $A_{\tilde\Gamma}$ is the algebra $A$ endowed with the action $\gamma \cdot a = \alpha_{\gamma,\sigma_\gamma}(a)$. In particular, if $G$ acts trivially on $A$, and any $\Gamma$-algebra can be seen like this, we get that 
$$ BC(\Gamma\times G, C_0(G,A) ) \Leftrightarrow BC(\Gamma , A),$$
where the action of $\Gamma\times G$ on $C_0(G,A)$ is given by 
$$((\gamma, g)\cdot f)(x) = \gamma\cdot (f(\gamma x g^{-1})).$$

The factorization of the assembly map via the partial assembly gives
\[BC^{(G)}(\Gamma,C_0(G,A)) \quad \& \quad BC(\Gamma,C_0(G,A)\rtimes_r \Gamma ) \implies BC(\Gamma\times G,C_0(G,A)),\]
for $A$ a $\Gamma$-algebra, seen as a $\Gamma\times G$-algebra via the trivial action of $G$.
\end{proof}

\begin{theorem}
Let $(\Lambda, \Gamma)$ be a Hecke pair, and $(G,K)$ its Schlichting completion, and $A$ a $\Gamma$-algebra. If 
\begin{itemize}
\item $G$ satisfies the Baum-Connes conjecture with coefficients in $C_0(G,A)\rtimes_r \Gamma$, 
\item every subgroup $L<G$ containing a conjugate of $\Lambda$ as a subgroup of finite index satisfies the Baum-Connes conjecture with coefficients in $A$, 
\end{itemize}
then $\Gamma$ satisfies the Baum-Connes conjecture with coefficients in $A$.
\end{theorem}

\begin{proof}
Since $G$ satisfies the Baum-Connes conjecture with coefficients in $C_0(G,A)\rtimes_r \Gamma$, the previous proposition ensures that it is enough to show that the partial assemby map 
$$\mu_{\Gamma, C_0(G,A)}^{(G)} : RK_\bullet^{\Gamma \times G}(\underline E\Gamma \times \underline E G, C_0(G,A)) \rightarrow RK_\bullet^G(\underline E G, C_0(G,A)\rtimes_r \Gamma) $$
is an isomorphism.


The space $\underline E G$ can be covered by open subset of the type $G\times_L U$, for $L$ a compact subgroup of $G$ and $U$ a $L$-space. Moreover, $G$ being totally disconnected, we can restrict to compact open subgroups $L$. By a standard Mayer-Vietoris argument, it is enough to show that 
\[\mu_{\Gamma , C_0(G,A)}^{(G)} : RK_\bullet^{\Gamma\times G}(\underline E \Gamma \times (G\times_L U) , C_0(G,A)) \rightarrow RK_\bullet^{G}(G\times_L U, C_0(G,A)\rtimes_r \Gamma)\]
is an isomorphism. 

By restriction principle, this is equivalent to show that
\[\mu_{\Gamma}^{(L)} : RK_\bullet^{\Gamma\times L}(\underline E \Gamma \times  U , C_0(G,A)_{|\Gamma\times L} ) \rightarrow RK_\bullet^{L}(U, (C_0(G,A)\rtimes_r \Gamma)_{|L})\]
is an isomorphism, i.e. $BC(\Gamma\times L,(C_0(G,A)\rtimes_r \Gamma)_{|F})$. 

Now, up to replacing $L$ by $L\cap K$, we can suppose $L<K$. As a $\Gamma \times L$-space, $G$ is isomorphic to $G/L \times L$, where the $L$ factor acts only on the right. Since $L$ and $K$ are compact open, the quotient is finite. Thus there are only finitely many $\Gamma\times L$-orbits: $[K:L]$ many. The typical stabilizer of an orbit is isomorphic to $H=\sigma^{-1}(L)$, so contains $\Lambda$ as a subgroup of finite index.

Green's isomorphism thus entails that $BC(\Gamma\times L,C_0(G,A)_{\Gamma\times L} )$ holds since we supposed $BC(H,A)$ for every such subgroup $H$.

We thus proved that 
\[\forall H \in \mathcal S_\Lambda , BC(H,A) \implies BC^{(G)}(\Gamma,C_0(G,A)) \]
and the proof is done. (We denoted by $\mathcal S_\Lambda$ the family of subgroups of $\Gamma$ containing $\Lambda$ as a subgroup of finite index.)
\end{proof}

The proof of theorem \ref{THM1} follows: by proposition \ref{CEHATM}, if $X$ admits a $\Gamma$-equivariant coarse embedding into a Hilbert space, $G$ is a-T-menable, hence satisfies the Baum-Connes conjecture with coefficients \cite{higson2001theory}. 

\section{Application to the Novikov conjecture} 

In order to prove the Novikov conjecture, we use Roe's \textit{descent principle} (see chapter 8 of \cite{john1996index}): to show that the Novikov conjecture for a discrete group $\Gamma$ holds, it is enough to construct a compact second-countable $\Gamma$-space $X$ such that 
\begin{itemize}
\item $\Gamma$ satisfies the Baum-Connes conjecture for every coefficients $C(X,A)$, for every $\Gamma$-algebra A;
\item $X$ is $F$-contractible, for every finite subgroup $F<\Gamma$.
\end{itemize}
This method was extensively used, originally by Higson (\cite{higson2000bivariant} 
) and later by Chabert, Echterhoff and Oyono-Oyono (\cite{chabert2004going}, theorem 1.9) and Skandalis-Tu-Yu (\cite{skandalis2002coarse}, theorem 6.1). They proved that the Novikov conjecture is satisfied if the discrete group admits a coarse embedding into a Hilbert space. We will proceed accordingly in the case of Hecke pairs where the subgroup and the coset space admits coarse embeddings into Hilbert spaces.

Let us denote by $\mathcal S_\Lambda$ the family of subgroups of $\Gamma$ containing $\Lambda$ as a subgroup of finite index. Recall that Deng proved the following result (section 4.1 of \cite{deng2019novikov}).

\begin{theorem}
Let $\Gamma$ be a group and $\Lambda$ a subgroup. Suppose $\Lambda$ is coarsely embeddable into a Hilbert space, then there exists a compact metrizable $\Gamma$-space $X$ such that, for every $L\in \mathcal S_N$, the restricted action of $L$ on $X$ is a-T-menable, and $X$ is $F$-contractible, for every finite subgroup $F<\Gamma$. 
\end{theorem}

We will need the following.

\begin{lemma}
Let $(G,K)$ be the Schlichting completion of a Hecke pair $(\Gamma , \Lambda )$. Suppose $G$ is a-T-menable, then there exists a second countable compact $G$-space $Y$ such that the action of $G$ is a-T-menable and $Y$, and $Y$ is $L$-contractible for every compact open subgroup $L<G$.
\end{lemma}

\begin{proof}
By theorem 5.4 of \cite{skandalis2002coarse}, since $X=\Gamma/\Lambda$ admits a coarse embedding into a Hilbert space, there exists a Hausdorff ample second countable groupoid $\mathcal G$ and an action of $\mathcal G$ on the Stone-\v{C}ech compactification $\beta X$ such that $\beta X\rtimes \mathcal G$ is a-T-menable. As the pseudogroup of $\beta X$ generated by the action of $G$ covers $\mathcal G(X)$ (in the terminology of \cite{skandalis2002coarse}, section 3.3), the ample groupoid $\mathcal G$ can be realized as the germs of the homeomorphisms induced by the action of $G$ on $\beta X$ (see \cite{skandalis2002coarse}, lemma 3.3), thus there exists a surjective continuous morphism $\beta X\rtimes G\rightarrow \beta X \rtimes \mathcal G$. Its kernel is a compact ample groupoid over $\beta X$: the isotropy groups are conjugates of $K$. Composing $\beta X\rtimes G\rightarrow \beta X \rtimes \mathcal G$ with the conditionally negative type proper continuous function given by a-T-menablity of $\beta X\rtimes \mathcal G$ gives that $\beta X\rtimes G$ is a-T-menable.

Moreover, we can actually make the same argument with $\beta X$ replaced by the base space $\mathcal G^0$ of $\mathcal G$, hence $\mathcal G^0\rtimes G$ is a-T-menable. As in \cite{higson2000bivariant} and \cite{skandalis2002coarse}, let $X$ be the space $Prob(\mathcal G^0)$ of Borel probability measures, endowed with the weak-$*$ topology: it is compact Hausdorff and second-countable. Lemma 6.7 of \cite{skandalis2002coarse} shows that $X\rtimes G$ is a-T-menable. The remaining assertion follows from the fact that $G$ acts on $X$ by affine isometries.
\end{proof}

Jean-Louis Tu proved in \cite{tu1999conjecture} that if $X$ is a second-countable compact $G$-space with an a-T-menable action of $G$, then $BC(G,C(X,A))$ holds for every $G$-algebra $A$. Combining this with Deng's result and the lemma above, if $(\Gamma,\Lambda)$ is a Hecke pair with $\Lambda$ and $\Gamma / \Lambda$ are coarsely embeddable into a Hilbert space, we know there exist:
\begin{itemize}
\item[$\bullet$] a second-countable compact metrizable $\Gamma$-space $X$ such that
\[BC(L,C(X, A)_{|L})\quad \forall L\in \mathcal S_N, \forall \Gamma\text{-algebra }A,\]
and $X$ is $F$-contractible for every finite subgroup $F<\Gamma$;
\item[$\bullet$] by a-T-menability of $G$, a second-countable compact metrizable $G$-space $Y$ such that $BC(G,C(Y, A))$ for all $G$-algebra $A$ and $Y$ is $L$-contractible for every compact subgroup $L<G$.
\end{itemize} 

We are now able to prove the main result of this section. It generalizes a result of Deng (theorem 1.1 of \cite{deng2019novikov}) to the case where the subgroup is not normal.

\begin{theorem}
Let $(\Gamma,\Lambda)$ be a Hecke pair such that $\Lambda$ and $\Gamma/ \Lambda$ are coarsely embeddable into a Hilbert space, then Novikov's conjecture holds for $\Gamma$.
\end{theorem}

\begin{proof}
It is enough to show that there exists a compact metrizable $\Gamma$-space $\Omega$ such that $\mu_{\Gamma,C(\Omega,A)}$ is an isomorphism for every $\Gamma$-$C^*$-algebra $A$.

Let $X$ and $Y$ be second-countable compact spaces as above, and let $\Omega= X\times Y$ with action of $\Gamma\times G$ given by $(\gamma , g)\cdot (x,y) = (\gamma\cdot x, g\cdot y)$. There is a $G$-equivariant isomorphism of $C^*$-algebras
$$ C_0(G,C(\Omega,A))\rtimes_r \Gamma \cong C(Y) \otimes (C_0(G\times X, A)\rtimes_r \Gamma )$$
that ensures that $\mu_{G,C_0(G,C(\Omega,A))\rtimes_r \Gamma }$ is an isomorphism. 
It is thus enough to show that the partial assembly map $\mu_{\Gamma , C_0(G,C(\Omega,A))}^{(\Gamma\times G)}$ is an isomorphism, which reduces to show that $\mu_{\Gamma , C_0(G,C(\Omega,A))}^{(\Gamma\times L)}$ is an isomorphism, for every compact open subgroup $L<G$, by standard restriction principle. With the same argument as before, the restricted action of $\Gamma\times L $ is a finite union of transitive actions with typical stabilizer $H= \sigma^{-1}(L)\in S_\Lambda$. By Green's principle, $\mu_{\Gamma , C_0(G,C(\Omega,A))}^{(\Gamma\times L)}$ is equivalent to $\mu_{H, C(X,A)}$. The latter being an isomorphism, this concludes the proof.   
\end{proof}

\section{Rational and $S$-integers points of algebraic groups over algebraic number fields} 
\label{EXPL}
We present two applications of \ref{THM1}. The first one recover known results, the second one is, to the author's knowledge, new.

\subsection{$SL_2$ of an algebraic number field and of $S$-integers}
The first application of theorem \ref{THM1} is to the groups $SL(2,\mathbb Z[\frac{1}{N}])$ and $SL(2,\mathbb Q)$. Both have the a-T-menable group $SL(2,\mathbb Z)$ as almost normal subgroup, and their respective Schlichting completion can be obtained by lemma \ref{LmWillis} with the homomorphisms $SL(2,\mathbb Z[\frac{1}{N}]) \rightarrow \prod_{p|N} PGL(2, \mathbb Q_p )$ and $SL(2,\mathbb Q)\rightarrow PGL(2,\mathbb A)$, where $\mathbb A$ is the ring of adèles. As both Schlichting completions are a-T-menable, the Hecke pairs are co-Haagerup. This generalizes easily to the following setting: let $F$ be a finite extension of $\mathbb Q$, $S$ a set of inequivalent valuations and $\mathcal O$ (respectively $\mathcal O_S$) the ring of integers (respectively $S$-integers) of $F$. Denote by $\mathbb A_{F}$ the ring of adèles of $F$.

\begin{corollary}
Let $S$ be a set of primes, and $\mathbb Z_S$ the ring of $S$-integers in $\mathbb Q$. Then $SL(2,\mathbb Z_S)$ and $SL(2,\mathbb Q)$ satisfy the Baum-Connes conjecture with coefficients. More generally, $SL(2,\mathcal O_S)$ and $SL(2, \mathbb A_{F})$ satisfy the Baum Connes conjecture with coefficients.
\end{corollary}


Let $G$ be an absolutely simple, simply connected algebraic group over $ F$.
\begin{corollary} 
Let $\Gamma$ be either $G(\mathcal O_S)$ or $G(F)$, then $\Gamma$ satisfies the Novikov conjecture. 
\end{corollary}

\begin{proof}
Let $A$ be a $\Gamma$-algebra. Denote $\mathbb A_{F}$ be the ring of adèles of $ F$ and $R$ its compact open ring of integers. Observe the almost normal subgroup $\Lambda = G(\mathcal O)$: its Schlichting completion $G$ will be that of $(G(\mathbb A_{ F}) , G(R)) $ if $\Gamma = G( F)$, or that of $(\prod_{\nu \in S} G( F_\nu) , \prod_{\nu \in S} G(\mathcal O_\nu) )$ if $\Gamma = G(\mathcal O_S)$. 

Every group $L<\Gamma$ containing $\Lambda$ as a subgroup of finite index is exact, so that $\mu_{L,B}$ is injective for every coefficients $B$.  

By theorem 5.2 of \cite{kasparov2003groups}, the map $\mu_{G, A}$ is injective. The diagram
\[\begin{tikzcd}
K^{top}_\bullet(\Gamma\times G, C_0(G,A)) \arrow{r}{\mu_{\Gamma\times G, C_0(G,A)}} & K_\bullet( C_0(G,A)\rtimes_r (\Gamma\times G)) \\
K^{top}_\bullet(\Gamma, A) \arrow{r}{\mu_{\Gamma, A}}\arrow{u}{\cong} & K_\bullet( A\rtimes_r \Gamma)\arrow{u}{\cong} \\
\end{tikzcd}\]
commutes, and by factorisation via partial assemby maps, we have 
$$\mu_{\Gamma\times G, C_0(G,A)}  = \mu_{G,C_0(G,A)\rtimes_r \Gamma}\circ \mu^{(G)}_{\Gamma , A},$$
hence the map $\mu_{\Gamma, A}$ is injective. We conclude by descent principle.
\end{proof}

These two corollaries also follow from \cite{guentner2005novikov}, where it is proven that, if $K$ is a field, every countable subgroup of $GL(n,K)$ is coarsely embeddable into Hilbert space, and if $n=2$, actually a-T-menable.

\subsection{Lattices in mixed product groups}%

Theorem \ref{THM1} and corollary \ref{COR1} allow to prove the Baum-Connes conjecture for groups in class $\mathcal C$, some of which are non-a-T-menable. For instance, the group $\mathbb Z^2\rtimes SL(2,\mathbb Z[\frac{1}{p}])$ is not Haagerup since $\mathbb Z^2 <\mathbb Z^2\rtimes SL(2,\mathbb Z)$ has relative property (T). Moreover, $\mathbb Z^2\rtimes SL(2,\mathbb Z)$ satisfies the Baum-Connes conjecture with coefficients, and is almost normal in $\mathbb Z^2\rtimes SL(2,\mathbb Z[\frac{1}{p}])$. The Schlichting completion of the pair is $PSL(2,\mathbb Q_p)$, which is a-T-menable so that we are in the conditions of the theorem. 

This result could have actually been proved by Oyono-Oyono's stability result of the Baum-Connes conjecture by extensions since it is a-T-menable by a-T-menable, or by Oyono-Oyono's result on group acting on trees, since it acts on the Bass-Serre tree of $SL(2,\mathbb Q_p)$ with stabilizers that are finite by a-T-menable. 

In order to show that the class $\mathcal C$ is interesting, we want to build examples of discrete groups in $\mathcal C$ that are not a-T-menable, or more generally, not an extension of a-T-menable by a-T-menable, nor hyperbolic, nor acting on trees with a-T-menable stabilizers.

\begin{proposition}
Let $\Gamma< G$ be an irreducible lattice in a product $G=G_1\times G_2$ of locally compact groups such that $G_1$ is not compact and has property (T), and $G_2$ is totally disconnected and a-T-menable. Then $\Gamma$ is not a-T-menable and, for every compact open subgroup $K<G_2$, $\Lambda = \varphi^{-1}(K)$ is co-Haagerup in $\Gamma$ (and thus almost normal). 
\end{proposition}

\begin{proof}
Let us show that $\Gamma$ is not a-T-menable. Since $\Gamma$ is of finite covolume in $G$, $L^\infty(G/\Gamma)$ admits a $G$-invariant state: it is co-F\o lner. By proposition 6.1.5 of \cite{cherix2001groups}, if $\Gamma$ was a-T-menable, so would $G$ be. Since $(G, G_1)$ has relative property (T) and $G_1$ is not compact, this is impossible. 

Consider the morphism $\varphi : \Gamma\rightarrow G_2$ given by the second projection: by irreducibility, we are in the situtation of lemma \ref{LmWillis}. This ensures that $\Lambda$ is almost normal in $\Gamma$, and that the Schlichting completion of $(\Gamma,\Lambda)$ is the quotient of $G_2$ by the largest normal subgroup contained in $K$, hence it is a-T-menable. Thus $\Gamma / \Lambda$ admits an equivariant coarse embedding into Hilbert space (equivalently $\Lambda$ is co-Haagerup in $\Gamma$).
\end{proof}


Let $G$ be an absolutely simple algebraic group over $\mathbb Q$, and $(\Gamma,\Lambda)$ be the Hecke pair $\left(G\left(\mathbb Z[\frac{1}{p}]\right), G(\mathbb Z)\right)$. Since the Schlichting completion identifies with that of $\left(G(\mathbb Q_p), G(\mathbb Z_p)\right)$, it is enough to know that $rk_{\mathbb Q_p} G(\mathbb Q_p)=1$ to know that the pair is co-Haagerup. We thus have the following.

\begin{corollary}
If all subgroups $L<\Gamma$ containing $\Lambda$ with finite index satisfy the Baum-Connes conjecture with coefficients, and $rk_{\mathbb Q_p} G(\mathbb Q_p)=1$, then $\Gamma$ satisfy the Baum-Connes conjecture with coefficients.
\end{corollary}



In general, the classification of rank $1$ groups over non archimedian local fields has been completed, and accounts can be found, for instance see \cite{tits1979reductive}. If one looks at groups with Tits index $C^2_{2,1}$ and $C^2_{3,1}$, choose a $\mathbb Q$-form $G$ such that $G(\mathbb R)$ is isomorphic to $Sp(n,1)$, with $n=3$ or $5$. Let $G(\mathcal O)$ be an arithmetic cocompact lattice: it is Gromov hyperbolic, thus any groups which contains it with finite index satisfies the Baum-Connes conjecture with coefficients by \cite{lafforgue2012conjecture}. Moreover, any $S$-arithmetic group $G(\mathcal O_S)$ will contain $\Lambda$ as a co-Haagerup almost normal subgroup, thus $\Gamma $ satisfies the Baum-Connes conjecture with coefficients. This gives example of countable subgroups of $GL(n,K)$, $n\leq 3$, that satisfy the Baum-Connes conjecture with coefficients.

In both these examples, $\Gamma$ is non a-T-menable, since it is an irreducible lattice in $G(\mathbb R)\times G(\mathbb Q_p)$, and by the $S$-arithmetic version of Margulis almost normal subgroup theorem, every normal subgroup of $\Gamma$ is either finite or commensurable to $\Gamma$, so that proving the Baum-Connes conjecture by expressing it as an extension will fail. Also, $\Gamma$ does not admit an isometric action on a tree with stabilizers satisfying the conjecture, nor is it hyperbolic. 



\bibliographystyle{plain}
\bibliography{biblio}

\end{document}